\documentclass[10pt]{article}
\usepackage{amsgen, amsmath, amsfonts, amsthm, amssymb, enumerate,amscd,graphics, amsfonts,dsfont,eulervm}
\usepackage[all]{xy}

\usepackage{graphicx}
\newtheorem{defn}{Definition}[section]

\newtheorem{thm}[defn]{Theorem}

\newcommand {\ZZ}{{\mathds Z}}

\newcommand {\KK}{{\mathcal K}}

\newcommand {\XX}{{\mathcal X}}

\newcommand {\A}{{\mathcal A}}

\newcommand {\G}{{\mathds  G}}

\newcommand {\CC}{{\mathcal C}}
\newcommand {\Q}{{\mathds Q}}
\newcommand {\R}{{\mathds R}}
\newcommand {\OO}{{\mathcal O}}

\newcommand {\HH}{{\mathcal  H}}

\newcommand {\M}{{\mathcal M}}

\newcommand {\CP}{{\mathds P}}

\newcommand {\D}{{\mathcal D}}
\newcommand {\QL}{{{\mathds Q}_{\ell}}}

\newcommand {\HCP}{{\hat{\mathbb P}}}

\newcommand{\barf}{{\bar{f}}}

\def\div{\operatorname{div}}

\def\dim{\operatorname{dim}}

\def\ord{\operatorname{ord}}

\title{Abelian Surfaces and the non-Archimedian Hodge-$\D$-conjecture - the semi-stable case. }
\author{Ramesh Sreekantan\\Indian Statistical Institute, Bangalore. }

\begin{document}
\baselineskip=15pt
\maketitle
\begin{abstract}

\end{abstract}

If $X$ is a smooth projective variety over $\R$, the Hodge $\D$-conjecture of Beilinson asserts the surjectivity of the regulator map to Deligne cohomology with real coefficients. It is known to be false in general but is true in some special cases like Abelian surfaces and $K3$-surfaces -- and still expected to be true when the variety is defined over a number field. We prove an analogue of this for Abelian surfaces  at a non-Archimedean place where the surface has bad reduction. Here the Deligne cohomology is replaced by a certain Chow group of the special fibre. The case of good reduction is harder and was first studied by Spiess \cite{spie} in the case of products of elliptic curve and by me in general \cite{sree4}. The case of bad reduction was also studied by me in \cite{sree2}. {\bf MSC Classification (2010) 19E15(11G25 14C15 14G20)}

\section{Introduction}

\subsection{The Hodge-$\D$-conjecture for Abelian surfaces}

Let $A$  be an Abelian surface over a $p$-adic local field $K$ and  $\A$ be a semi-stable model over the ring of integers $\OO=\OO_K$. Let  $\A_v$ be the  special fibre over the closed point $v$ - which we assume is semi-stable, namely a union of divisors with normal crossings whose components are smooth.  The aim of this paper is to prove a non-Archimedean analogue of the Hodge-$\D$-conjecture for such Abelian surfaces. 

This conjecture states that the map 
$$CH^2(A,1) \otimes \Q \stackrel{\partial}{\longrightarrow}  PCH^1(\A_v) \otimes \Q$$ 
is surjective. Here $PCH^1(\A_v)$ is a certain sub-quotient of the Chow group of the special fibre.  This group  has the property that 
$$\dim_{\Q} PCH^1(\A_v)\otimes \Q=-\ord_{s=1} L_p (H^2(A),s)$$
where $L_p(H^2(A),s)$ is the local $L$-factor at $p$. This group can hence  be viewed as a $p$-adic version of the Real Deligne cohomology -- which has that property with respect to the Archimedean factor -- and hence the map $\partial$ can be viewed as a $p$-adic version of the regulator map. 

When $\A_v$ is smooth, that is, $p$ is a prime of good reduction, the group $PCH^1(\A_v)$ is simply $CH^1(\A_v)$. This case was studied in \cite{spie} and \cite{sree4}. When $A$ is a product of (non-isogenous) elliptic curves  and $p$ is a prime of semi-stable reduction for both this was studied in \cite{sree2}. This paper essentially closes the chapter --- proving it in the remaining  case of semi-stable reduction of simple Abelian surfaces. 

Note that the group $CH^2(\A,1)\otimes \Q$ has many diffrent avatars - it is the same as the $\KK$-cohomology group $H^1(\A,\KK_2)$ and the motivic cohomology group $H^2_{\M}(\A,\Q(2))$.

\vspace{\baselineskip}

 \noindent {\em Acknowledgements:} This work was largely done when the author was visiting CRM in Montreal some years ago. I would like to thank CRM for their hospitality and ISI for their support while this work was done.

\section{The target space of the boundary map}

Let $X$ be a smooth proper variety over a local field $K$ and $\OO$
the ring of integers of  $K$ with closed point $v$ and generic point
$\eta$.

By  a model $\XX$ of $X$ we mean a flat proper scheme $\XX
\rightarrow Spec(\OO)$ together with an isomorphism of the generic
fibre $X_{\eta}$ with $X$. Let $Y$ be the special fibre $\XX_v=\XX
\times Spec(k(v))$ and let $i:Y \hookrightarrow \XX$ denote the inclusion map. 
We will always also make the assumption that the
model is strictly semi-stable, which means that it is a regular
model and the fibre $Y$ is a divisor with normal crossings whose
irreducible components are smooth, have multiplicity one and
intersect transversally.

\subsection{Consani's Double Complex}

In \cite{cons}, Consani defined a double complex of Chow groups of
the components of the special fibre with a monodromy operator $N$,
following the work of Steenbrink \cite{stee} and
Bloch-Gillet-Soul\'{e} \cite{bgs}. Using this complex she was able
to relate the higher Chow group of the special fibre at a
semi-stable prime to the regular Chow groups of the components. This
relation is what is used in defining the group $PCH$.

Let $Y=\bigcup_{i=1}^{t} Y_i$ be the special fibre of dim $n$ with
$Y_i$ its irreducible components. For $I \subset \{1,\dots,t\}$,
define
$$Y_I= \cap_{i \in I} Y_i$$
Let $r=|I|$ denote the cardinality of $I$. Define
$$Y^{(r)}:=\begin{cases} \XX & \text{ if } r=0 \\ \coprod_{|I|=r} Y_{I}& \text{ if } 1 \leq r \leq n \\
\emptyset & \text { if } r>n \end{cases}$$
For $u$ and $t$ with $1 \leq u \leq t < r$ define the map
$$\delta(u):Y^{(t+1)} \rightarrow Y^{(t)}$$
as       follows.        Let       $I=(i_1,\dots,i_{t+1})$ with
$i_1<i_2<...<i_{t+1}$. Let $J=I-\{i_u\}$. This gives an embedding
$Y_I \rightarrow Y_J$. Putting these  together induces the map
$\delta(u)$. Let $\delta(u)_*$  and $\delta(u)^*$ denote the
corresponding maps on Chow homology  and cohomology respectively.
They further  induce the Gysin and restriction maps on the Chow
groups.

Define
$$\gamma:=\sum_{u=1}^{r+1} (-1)^{u-1} \delta(u)_*$$

and
$$ \rho:=\sum_{u=1}^{r+1} (-1)^{u-1} \delta(u)^*$$
These maps have the properties that
\begin{itemize}
\item $\gamma^2=0$ \item $\rho^2=0$ \item $\gamma \cdot \rho +
\rho \cdot \gamma =0$
\end{itemize}

\subsection{The group PCH} Let $a,q$ be two integers with $q-2a >0$.
$$PCH^{q-a-1}(Y,q-2a-1):= \begin{cases} \displaystyle{ \frac{Ker(i^*i_*:CH_{n-a}(Y^{(1)}) \rightarrow CH^{a+1}(Y^{(1)}))}{Im(\gamma:CH_{n-a}(Y^{(2)}) \rightarrow CH_{n-a}(Y^{(1)}))}\otimes \Q }& \text{if $q-2a=1$}\\
\\   \displaystyle {\frac{Ker(\gamma:CH_{n-(q-a-1)}(Y^{(q-2a)}) \rightarrow CH_{n-(q-a-1)}(Y^{(q-2a-1)}))}
{Im(\gamma:CH_{n-(q-a-1)}(Y^{(q-2a+1)}) \rightarrow
CH_{n-(q-a-1)}(Y^{(q-2a)}))}\otimes \Q } & \text {if $q-2a>1$}
\end{cases}
$$
Here $n$ is the dimension of $Y$. Note that if $q-2a>1$ and $Y$ is
non-singular, this group is $0$, while if $Y$ is singular and
semi-stable, the Parshin-Soul\'{e} conjecture implies that this
group is $CH^{q-a-1}(Y,q-2a-1)\otimes \Q$. If $q-2a=1$ and $Y$ is
non-singular, the group is $CH^a(Y) \otimes \Q$. Our interest is in
the remaining case, namely when $q-2a=1$ and $Y$ is singular.

The `Real' Deligne cohomology has the property that its dimension is
the order of the pole  of the Archimedean factor of the $L$-function
at a certain point on the left of the critical point. The group
$PCH^1(Y)$ is expected to have a  similar property. Let $F^*$ be the geometric
Frobenius and $N(v)$ the number of elements of $k(v)$. The local
$L$-factor of the $(q-1)^{st}$-cohomology group is then
$$L_v(H^{q-1}(X),s)=(det(\textsl{I}-F^*N(v)^{-s}|H^{q-1}(\bar{X},\QL)^{I}))^{-1}$$
\begin{thm}[Consani] Let $v$ be a place of semistable reduction.
Assuming the weight-monodromy conjecture, the Tate conjecture for
the components   and the injectivity of the cycle class map on the
components $Y_I$, the Parshin-Soul\'{e} conjecture and that $F^*$
acts semisimply on $H^*(\bar{X},\QL)^{I}$. we have
$$\dim_{\Q} PCH^{q-a-1}(Y,q-2a-1)=-\ord_{s=a} L_v(H^{q-1}(X),s):=d_v$$
\end{thm}

\begin{proof} \cite{cons}, Cor 3.6.
\end{proof}
From this point of view the group $PCH^{q-a-1}(Y,q-2a-1)$ can be
viewed as a non-Archimedean analogue of the `Real' Deligne
cohomology. Since the $L$-factor at a prime of good reduction does
not have a pole at $s=a$ when $q-2a>1$, the Parshin-Soul\'{e}
conjecture can be interpreted as the statement that this
non-Archimedean Deligne cohomology has the correct dimension, namely
$0$, even at a prime of good reduction.

As is clear from the definition, the group $PCH$ depends on the
choice of the semi-stable model of $X$. However, Consani's theorem
says that the dimension does not. So to a large extent one can work
with any semi-stable model. Perhaps the correct definition is one
obtained by taking a limit of semi-stable models as in the work of
Bloch, Gillete and Soul\'{e} \cite{bgs} on non-Archimedean Arakelov
theory.

\subsection{Elements of the higher Chow group}

From this point on we specialize to the case when $X$ is a surface
and further $n=2$, $q=3$ and $a=1$. We will be interested in group
$CH^2(X,1)$ and the map to $PCH^1(Y):=PCH^1(Y,0)$. This is related
to the order of the pole of the $L$-function of $H^2(X)$ at $s=1$.
Soon we will further specialize to the case when $X$ is an Abelian surface.

Let $X$ be a surface over a field $K$.  The group  $CH^2(X,1)$ has the following presentation \cite{rama}.
It is generated by formal sums of the type
$$\sum_i (C_i,f_i)$$
where $C_i$ are curves on $X$ and $f_i$ are $\bar{K}$-valued
functions on the $C_i$ satisfying the cocycle condition
$$\sum_i \div{f_i}=0.$$
Relations in this group are give by the tame symbol of pairs of
functions on $X$.

There are some  elements of this group coming from the
product structure
$$CH^1(X_L) \otimes CH^1(X_L,1) \longrightarrow CH^2(X_L,1) \stackrel{Nm^L_K}{\longrightarrow} CH^2(X,1)$$
Here $Nm^L_K$ is the norm map from a finite extension $L$ of $K$. The image of this group as $L$ runs through all finite extensions of $K$ is called the subgroup of {\em decomposable elements},  $CH^2_{dec}(X,1)$ 

A theorem of Bloch \cite{bloc} says that $CH^1(X_L,1)$ is simply
$L^{*}$ where $L$ is the field of definition of $X_L$ so such an
element looks like a sum of elements of the type $(C,a)$ where $C$
is a curve on $X_L$ and $a$ is in $L^*$. The group of {\em indecomposable elements} is the quotient group 
$$CH^2(X,1)_{ind}=CH^2(X,1)/CH^2_{dec}(X,1).$$
In general, it is hard to find elements in this group.  

The group $CH^2(X,1)\otimes \Q$ has several avatars --- it is the  same as the $\KK$-cohomology
group $H^1_{Zar}(X,\KK_{2})\otimes \Q$ and the motivic cohomology
group  $H^3_{\M}(X,\Q(2))$.

\subsection{The boundary map}

The usual Beilinson regulator maps the higher Chow group to the Real
Deligne cohomology. In the non-Archimedean context, it appears that
the boundary map
$$\partial:CH^2(X,1) \longrightarrow PCH^1(Y)$$
plays a similar role. It is defined as follows
$$\partial\left( \sum_i (C_i,f_i)\right)=\sum_i \div_{\bar{C_i}}(\bar{f_i})$$
where $\bar{f_i}$ is the function $f_i$ on the closure $\bar{C_i}$  of $C_i$ in the semi-stable
model $\XX$ of $X$. By the cocycle condition, the `horizontal
divisor', namely, the closure  $\sum_i \overline{\div_{C_i}(f_i)}$ of $\sum_i \div_{C_i}(f_i)$, is $0$ 
and so the boundary is supported on the special fibre. Further, since the
boundary $\partial$ of an element is the sum of divisors of
functions, it lies in $Ker(i^*i_{*})$.

For a decomposable element of the form $(C,a)$ the boundary map is
particularly simple to compute,
$$\partial((C,a))=\ord_v(a)\CC_v$$
Where $\CC_v$ is the special fibre of a model $\CC$. In particular, a cycle in the special fibre which is not the restriction of the closure of a cycle in the generic fibre cannot appear in the boundary of the subgroup of decomposable elements. 

\section{Semi-stable reduction of Abelian surfaces}

\subsection{Types of semi-stable reductions}

If $A$ is an abelian surface over a local field, Kulikov and  Persson-Pinkham classified the possible semi-stable degenerations. For a surface $X$ the {\bf  dual graph} of its special fibre is defined as follows. It is the simplicial complex with one vertex $v_i$ for every component $Y_i$. The simplex $[v_{i_1},\dots, v_{i_k}]$ lies in the simplicial complex if and only if $Y_{i_1}\cap Y_{i_2} \cap \dots \cap Y_{i_k} \neq \emptyset$. 
\begin{thm}[Kulikov-Persson-Pinkham] \cite{Fr-Mo}If $\A$ is the Neron model of an Abelian surface --- so that $K_{\A}=0$ --- the possible semi-stable special fibres are 
\begin{itemize}
\item Type 1 --  $\A_p$ is smooth and the dual graph is a point. \\
\item Type 2 -- $\A_p=\bigcup_i Y_i$ where $\bigcup_i Y_i$ is a cycle of elliptic ruled surfaces such that adjacent surfaces $Y_i \cap Y_j$ intersect at elliptic curves $E_{ij}$. All the elliptic surfaces are isomorphic. The dual graph is $S^1$.\\ 
\item Type 3 -- $\A_p$ is a cycle of rational surfaces, each isomorphic to $\CP^1 \times \CP^1$ such that the dual graph is topologically $S^1 \times S^1$.  The double curves are `$-1$ hexagons' --  there are six components in every double curve and  each component is a $-1$ curve. 
\end{itemize}
\end{thm}

\subsection{The group $PCH^1(\A_p)$.}

We want to study the boundary of map
$$CH^2(A,1)\otimes \Q \stackrel{\partial}{\longrightarrow} PCH^1(\A_p)\otimes \Q$$
In this case the target space is 
$$PCH^{1}(\A_p)=\frac{Ker(i^*i_*:CH_{1}(Y^{(1)})
\rightarrow CH^{2}(Y^{(1)}))}{Im(\gamma:CH_{1}(Y^{(2)})
\rightarrow CH_{1}(Y^{(1)}))}\otimes \Q$$
with dimension 
$$\dim_{\Q} PCH^{1}(\A_p)=-\ord_{s=1} L_v(H^{2}(A),s):=d_v$$
We have to study each case separately. 

\subsubsection{Type 1 degenerations of Abelian surfaces}

This is the case when the special fibre is a smooth  abelian surface and  was  studied in \cite{sree4}. Here the target space of the boundary map is simply $CH^1(\A_p)\otimes \Q$ and the rank of this space is at least $2$. We showed that there exists a new element of the higher Chow group of $A$ for each new element of $CH^1(\A_p)$.  The argument here is quite subtle and uses deformation theory. This includes the case when the special fibre is a product of two elliptic curves.

\subsubsection{Type 2 degenerations of Abelian surfaces}

In this case the special fibre is a cycle of elliptic ruled surfaces $Y_i$. These surfaces intersect at elliptic curves 
$$Y_i \cap Y_j=\begin{cases} E_{ij} \text{ an elliptic curve } & j=i\pm 1\\
\emptyset & |j-i|>1 \end{cases}
$$
which are the bases of the elliptic ruled surfaces. All the elliptic curves are isomorphic. Here we have 
$$Y^{(j)}=\begin{cases}\A_p & j=0\\
 \bigsqcup_i  Y_i & j=1\\
\bigsqcup_{i} Y_{i} \cap Y_{(i+1)}=\bigsqcup_i E_{i(i+1)} & j=2 \\
\emptyset & j>2\end{cases}$$
As the conjectures upon which it is conditional hold in this case, we can apply Consani's theorem. Hence the dimension of the group $PCH^1(\A_p) \otimes \Q$ is the order of vanishing of the local $L$-factor at $v$ of the $L$-function of $H^2(A)$. This  dimension can be computed  using the analogue of the Clemens-Schmidt exact sequence due to \cite{bgs} and turns out to be $2$. Hence we need to construct $2$ higher Chow cycles. 

One of them can be constructed as follows.  If $\D$ is a cycle in $CH_1(\A)$ then the restriction of $\D$ to $\A_p$, $\D_p$ lies in $PCH^1(\A_p)$. This is because 
$$i^*i_*(\D_p)=i^*(\D)$$
and 
$$i^*(\D)=\oplus_j (\D \cap \div(\pi))|_{Y_j}.$$
so $\D \cap \div(\pi)$ is the divisor of the function $\pi$ restricted to $\D$  hence is $0$ in $CH^2(\A)$ and so maps to  $0$ in $CH^2(Y^{(1)})$. Hence the restriction of a generic cycle always lies in the group $PCH^1(\A_p)$.  This bounds one of the generators of the group $PCH^1(\A_p)$. The conjecture predicts that there is second element of the higher Chow group which bounds the other generator of this group. We will construct this cycle in the next section.

\subsubsection{Type 3 degenerations of Abelian surfaces}

In this case the individual components are $\CP^2$ blown up at the vertices of a triangle - which we will denote by $\HCP^2$. This results in a `-1 hexagon' where three of the six sides are the strict transforms of the edges of the triangle and the other three are the exceptional fibres. They are all (-1)-curves on $\HCP^2$ and if two components intersect,  they  intersect along one of these curves. 
$$Y^{(r)}=\begin{cases}\A_p & r=0\\
 \bigsqcup_i  Y_i= \bigsqcup \HCP^2_i & r=1\\
\bigsqcup_{i} Y_{i} \cap Y_{j}=\bigsqcup_{i,j} \CP^1_{i,j} & r=2 \\
\bigsqcup_{i,j,k} Y_{i,j,k}=\bigsqcup_{i,j,k} \CP^0_{i,j,k} & r=3\\
\emptyset & r>3
\end{cases}$$
Here one knows from the analogue of the Clemens-Schmidt exact sequence that the dimension is $3$.  Hence the conjecture predicts that there are three elements of the higher Chow group. One of them is the boundary of a decomposable element coming from the genus two curve on the generic fibre. We will show that there are at least two other elements which are linearly independent. 

\section{Higher Chow cycles}

\subsection{Collino's construction}

Collino \cite{coll} constructed a higher Chow cycle on a principally polarized  Abelian surface $A$ as follows.  Since $A$ is principally polarized, $A=Jac(C)$ 
where  $C$ be a genus $2$ curve. Let $P$ and $Q$ be two ramification points on $C$. There is a function $f$ on $C$ with divisor 
$$ \div(f)=2P-2Q$$
Let $C_P$ and $C_Q$ be the images of the curve $C$ under the maps $\iota_P$ and $\iota_Q$ where 
$$\iota_x(y)=y-x$$
and let $f_P$ and $f_Q$ be the function $f$ being thought of as a function on $C_P$ and $C_Q$ respectively. Then 
$$ \div(f_P)=2(0)-2(Q-P) \text{ and } \div(f_Q)=2(P-Q)-2(0)$$
Since $P-Q$ is a two torsion point on the Abelian surface, $P-Q=Q-P$. Hence the element 
$$\Xi_{P,Q}=(C_P,f_P)+(C_Q,f_Q)$$
satisfies the co-cycle condition 
$$\div(f_P)+\div(f_Q)=0$$
hence is an element of the higher Chow group $CH^2(A,1)$.

\subsection{Surjectivity}

Our conjecture states that the boundary map in the localization sequence is surjective. We show that the element of Collino's described above, with suitable choices of points $P$ and $Q$, suffices to show surjectivity in the cases when the special fibre of the Abelian surface is singular, as well as in the case when the special fibre is a product of elliptic curves. 

We do this by computing the boundary of the element in terms of the components of the regular minimal model of the curve $C$. For that we need the theorems of Parshin \cite{pars} on minimal models of genus $2$ curves.  In all the cases, the computation of the boundary is done as follows. Suppose the special fibre 
$$\CC_p=\bigcup_{i} X_i.$$
Then 
$$\div(\bar{f})={\mathcal H}+ \sum_i a_i X_i$$
where $\barf$ is the function $f$ on the closure $\CC$ and  ${\mathcal H}$ is the horizontal divisor $\overline{\div(f)}$. To compute the $a_j$ we do the following. We know that the decomposable element $(C,p^k)$ has boundary $\div(p^k)=k \CC_p=k(\sum_i  X_i)$. Hence 
	$$\div(\bar{f})-\div(p^{a_j})={\mathcal H}+\sum_i (a_i-a_j)X_i$$ 
	and in particular, $X_j$ is not in the support. The degree of a divisor of a function on a curve on an algebraic surface which is not contained in the support is $0$. Hence restricting this to $X_j$ gives us an equation 
	$$({\mathcal H}.X_j)+ \sum_i (a_i-a_j)(X_i.X_j)=0$$
	Using that and what we know about the intersection numbers $(X_i.X_j)$ gives us a linear equation among the $a_i$ not including $a_j$
	
	However, we can simplify our calculations using the following observation. If $X$ is a component of the special fibre $\CC_p$ then $(X.\CC_p)=0$. So equivalently we have the equation 
	$$({\mathcal H}.X_j)+ \sum_i a_i(X_i.X_j)=0$$
	though here we have to use what we know abou the self intersection $(X_j.X_j)$.
	
	 Repeating this with the different components gives  us a system of simultaneous equations in the $a_i$ which we can solve quite easily. We get as many equations as components this way and so the space of solutions is one dimensional. Sometimes it is convenient to make a choice of the coefficient of one of the components in order to get a `nice' description of the boundary. 
	 
From the N\'{e}ron mapping property the map $\iota_x:C \rightarrow A$ extends to a map, which we will also denote by  $\iota_x$, 
$$\iota_x:\CC^{ns} \rightarrow   \A^0$$
 where $\CC^{ns}$ is the curve $\CC$ with the singular points removed and $\A^0$ is the N\'{e}ron model of the Jacobian.  The special fibre of the N\'{e}ron model of the Jacobian is the group  $\A_p \backslash \bigcup_{i,j,i \neq j} (Y_i \cap Y_j)$. where $\A_p=\bigcup Y_i$ is the special fibre of a minimal regular model with components $Y_i$.   Each component of the special fibre is an extension of an abelian variety by a power of $\G_m$ -- so in our case it is either an abelian surface or an extension of an elliptic curve $E$ by $\G_m$ or $\G_m \times \G_m$. 
We chose a particular component where the closure of the zero section lies and define that to be the identity component.  The set of components has the structure of  a finite abelian group.

The element in the higher Chow group that we consider is $\Xi_{P,Q}$. The boundary of the element  $\Xi_{P,Q}$ is the closure of 
$$\iota_P(\div(\barf)+\iota_Q(\div(\barf)$$ 
in the special fibre $\A_p$.  Since the horizontal cycles cancel, one has 
$$\partial(\Xi_{P,Q})=\sum_i  a_i (\iota_P(X_i) + \iota_Q(X_i))$$
The curves $\iota_P(X_i)$ and $\iota_Q(X_i)$ are linearly equivalent in $PCH^1(Y)$ hence the boundary is 
$$\partial(\Xi_{P,Q})=\sum_i  2 a_i (\iota_P(X_i) )$$
in $PCH^1(Y)\otimes \Q$.

Hence what we have to show is that for a suitable choice of $P$ and $Q$ one obtains a new cycle in $PCH^1(Y)$ and in the case of Type 3, we show that for different choices of $P$ and $Q$ we can get two new cycles. 

We now do a case by case analysis. There are seven cases of minimal regular models of genus two curves. In all that follows let $\CC$ be in the minimal regular model of a genus $2$ curve $C$ and $\CC_p$ the special fibre. We use the notation of \cite{pars}.  In the pictures, the bold lines correspond to the curves and the thin lines indicate where the Weierstrass points lie.

\begin{table}[h!]
	\begin{center}
		
		\label{tab:table1}
		\begin{tabular}{l|c|c} 
			\textbf{Case} & \textbf{Type of Jacobian} & \textbf{Rank of $PCH^1$}\\
		    \hline
			I & 1 & \ $\geq$ 2\\
			II & 2 & 2\\
			III & 3& 3\\
			IV & 1 & $\geq$ 3\\
			V & 2 & 2 \\
			VI & 3 & 3 \\
			VII & 3 & 3 \\
			\hline 
		\end{tabular}
	\end{center}
\end{table}

\subsubsection{Case I}
$$
\includegraphics[width=80mm]{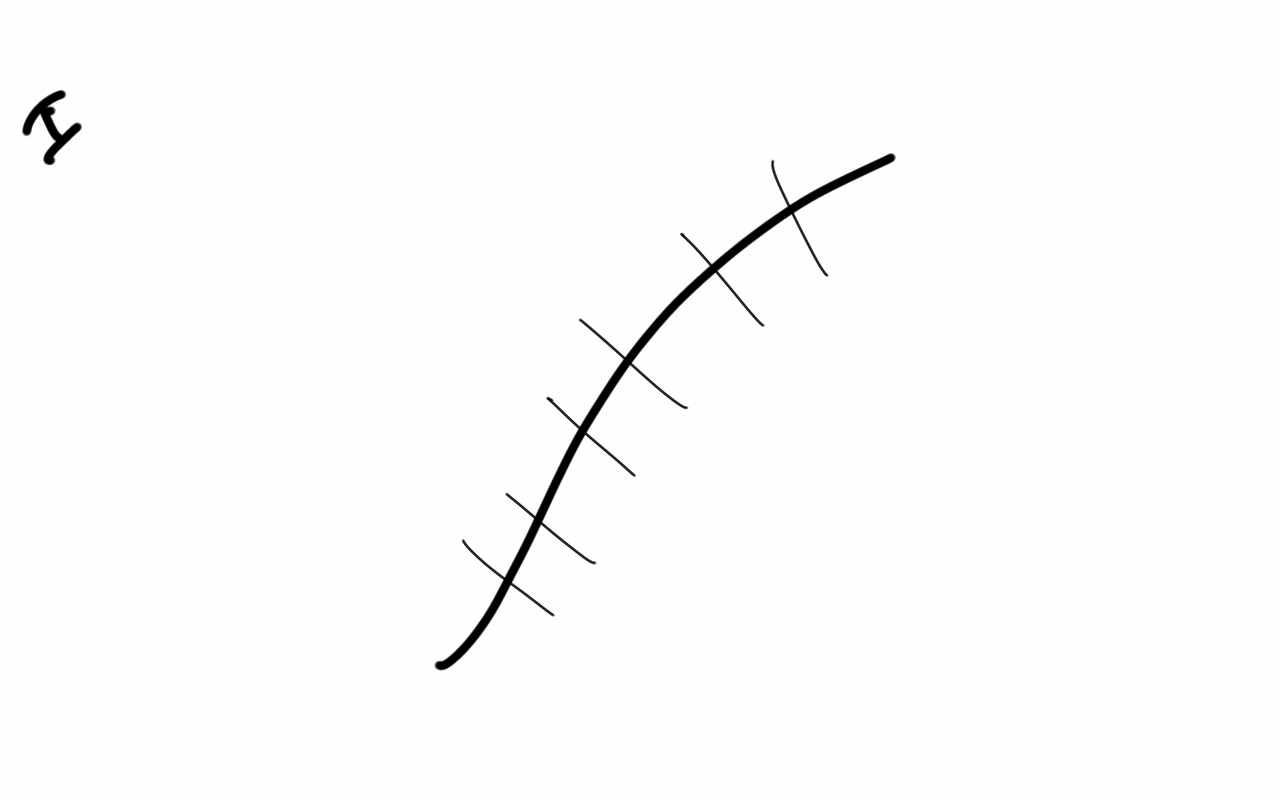}
$$
In this case the curve $\CC$ reduces to a smooth genus $2$ curve and the Jacobian is of Type 1. Since the special fibre has only one component one has 
$\div(\bar{f}_P)= \HH + a \CC_p$
Computing the intersection with $\CC_p$ shows that $a=0$. Hence Collino's cycle has no boundary here --- but the decomposable element can be used to bound $\CC_p$. However, the dimension of $PCH^1(Y)\otimes \Q=CH^1(Y)\otimes \Q$ is at least $2$ owing to the existance of the Frobenius endomorphism. The conjecture predicts there are at least {\em two} higher Chow cycles. Further, it is usually the case that the Picard number of the special fibre is strictly larger than that of the generic fibre. In those cases decomposable cycles will not suffice to prove surjectivity. Collino's cycle has boundary $0$ so does not work. Hence we have to find new indecomposable cycles. This is the content of \cite{sree4}. The idea there was to `deform' a rational curve corresponding to the extra cycle to construct a new element. Curiously this is the hardest case.

\subsubsection{Case II} 
$$
\includegraphics[width=80mm]{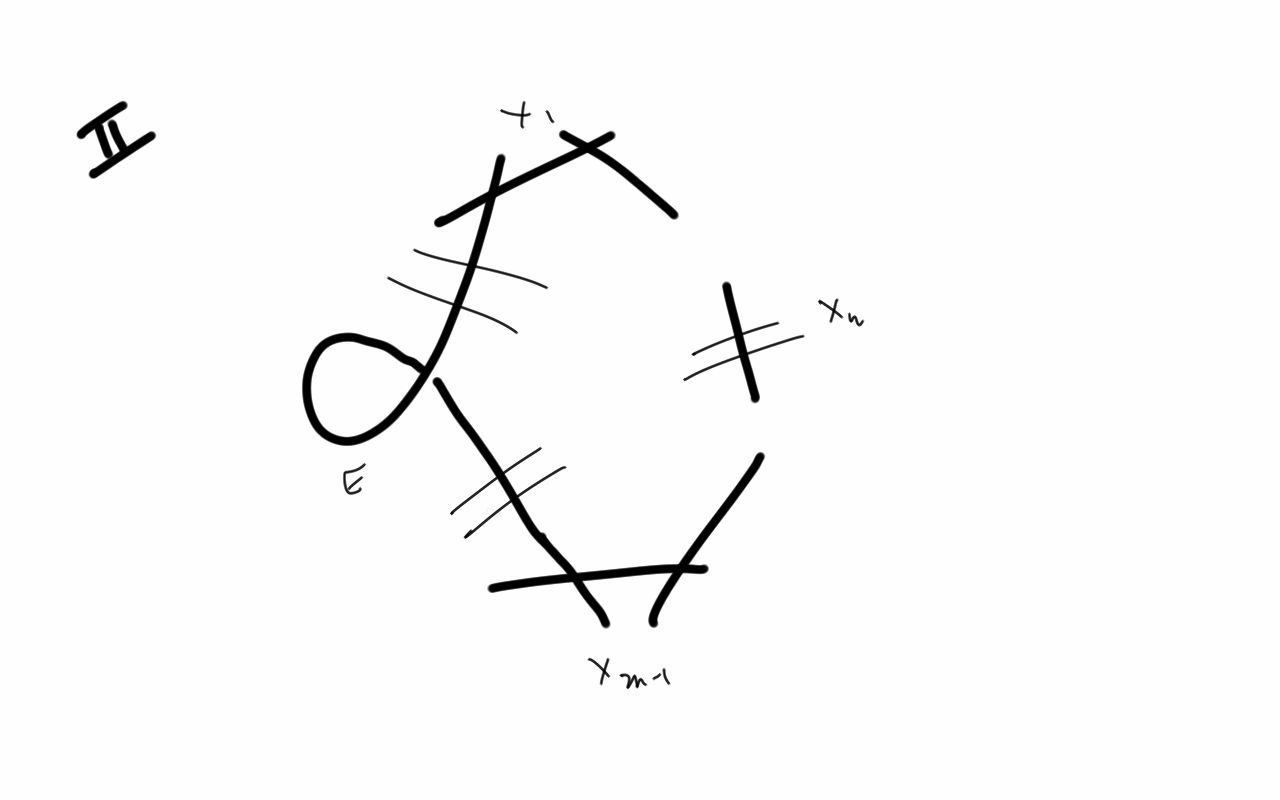}
$$
In this case the stable model of the curve is a genus $1$ curve with a node and the Jacobian is of Type 2. The special fibre $\CC_p$ of the regular minimal model consists of a genus $1$ curve $E$ and a chain of $(2n-1)$ rational curves $X_i$ meeting  $E$ at two points $\alpha$ and $\beta$,  where $n\geq 2$ is an integer. One has $X_i^2=-2=E^2$.  Finally, the closure of four  of the Weierstrass points meets $E$ and the remaining two  meet the middle component $X_{n}$. 

Choose $P$ and $Q$ such that $P$ meets $E$ and $Q$ meets $X_n$.  One has 
$$\div(\bar{f}_P)=\HH+aE+\sum_{i=1}^{2n-1} b_i X_i$$
To compute $a$ and $b_i$ we modify by a decomposable element and  intersect with $E$ and $X_i$. As we remarked we can simply consider the restriction  to $E$ or $X_i$. For reasons of symmetry one has 
$$b_i=b_{2n-i}$$
$$0=(\div(\barf_P).E) -=2-2a+b_1+b_{2n-1}=2-2a+2b_1  \Rightarrow   b_1=a-1$$
$$0=(\div(\barf_P).X_1)=a-2b_1+b_2 \Rightarrow b_2=b_1-1=a-1-1=a-2$$
continuing in this manner one can see that 
$$b_k=a-k\hspace{2cm}  1 \leq k \leq (n-1)$$
and by symmetry 
$$b_{2n-k}=b_k \hspace{2cm} 1 \leq k \leq (n-1)$$
Finally 
$$0=-2+2b_{n-1}-2b_n= -2+2(a-(n-1))-2b_n     \Rightarrow b_n=a-n$$
So one has 
$$\div(\barf_P)=\HH+aE- \left( \sum_{i=1}^{n-1} (i-a)(X_i+X_{2n-i}) + (n-a)X_n\right )$$
A similar calculation shows that 
$$\div(\barf_Q)=-\HH+aE-\left( \sum_{i=1}^{n-1} (i-a)(X_i+X_{2n-i})  + (n-a)X_n\right)$$
so combining these two, the boundary of the element $\Xi_{P,Q}$ in the N\'{e}ron special fibre is 
$$\partial (\Xi_{P,Q})=\iota_P(\div(\barf_P))+\iota_Q(\div(\barf_Q))=2 a\iota_P(E)- 2\left( \sum_{i=1}^{n-1} (i-a)(\iota_P(X_i) + \iota_P(X_{2n-i}))) + (n-a)\iota_P(X_n) \right)$$
A different choice of Weierstrass points will either change sign, if  the roles of $P$ and $Q$ are reversed,  or  have boundary $0$, if $P$ and $Q$ lie on the same component. 

Each component of the  N\'{e}ron special fibre is a non-split  extension of $E$ by $\G_m$ and the group of components is isomorphic to $\ZZ/(2n-1)\ZZ$. Each $X_i$ is isomorphic to $\G_m$ and its closure in the special fibre of the degenerate Abelian surface is a $\CP^1$. The special fibre of the closure of the curve $C_P$ is a copy of $E$ in one component with a chain of $\CP^1$s meeting $E$ at two different points which are translates of each other. Hence the boundary  of the decomposable element $(C,p^a)$ is 
$$\partial((C,p^a))=a\CC_p = a  \left(\iota_P(E) + \sum_{i=1}^{2n-1} \iota_P(X_i) \right)$$
Adding twice this to our computation of the boundary of $\Xi_{P,Q}$ gives that, up to a decomposable element, the boundary of $\Xi_{P,Q}$ in $PCH^1(\A_p)) \otimes \Q$ is
$$\partial(\Xi_{P,Q})=- 2\left( \sum_{i=1}^{n-1} i (\iota_P(X_i) + \iota_P(X_{2n-i})) + n ( \iota_P(X_n)) \right).$$
This can be seen to be non-zero by intersecting with $E$, for instance. In particular, it is not a multiple of $\CC_p$. Hence the two elements $C_p$ and $\partial(\Xi_{P,Q})$  are linearly independent and therefore generate $PCH^1(Y)\otimes \Q$.

\subsubsection{Case III}
$$
\includegraphics[width=80mm]{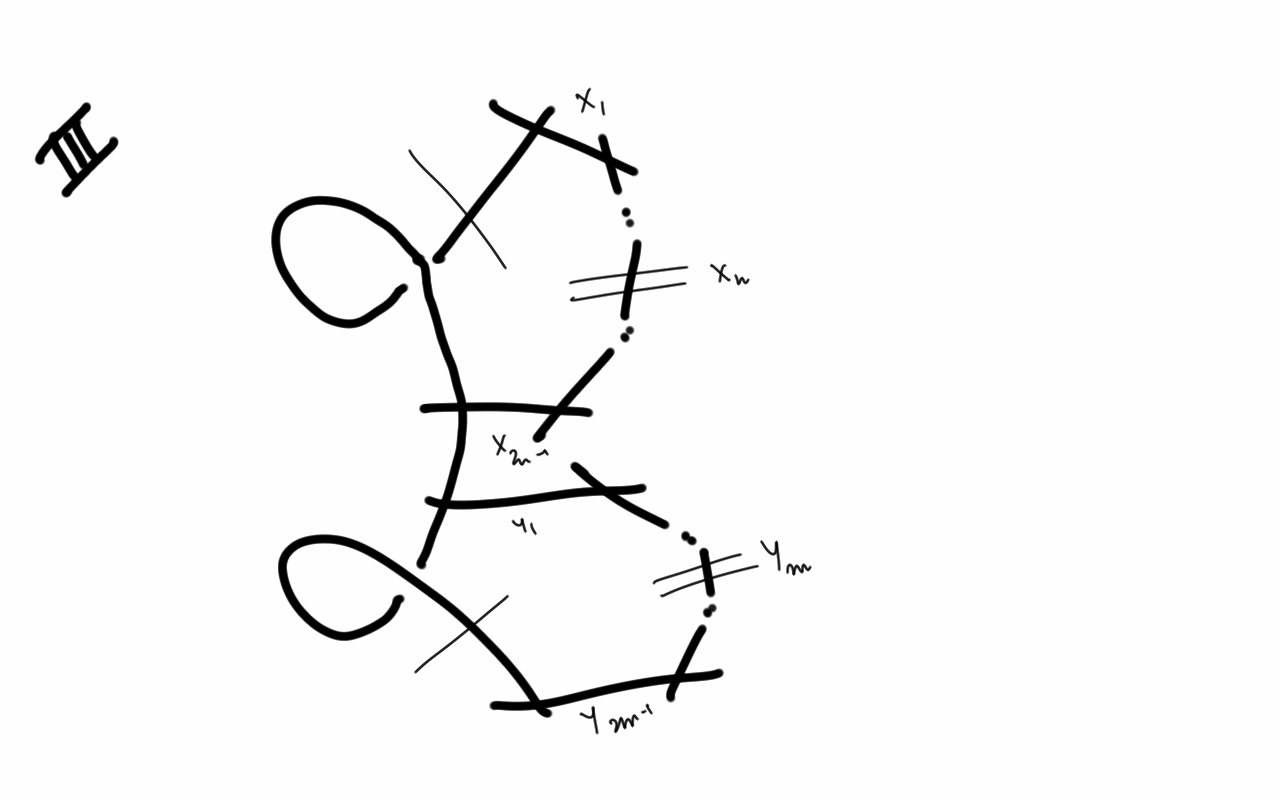}
$$
In this case the stable model is a genus $0$ curve with two nodes and the Jacobian is of Type 3. This can be viewed as the case when the elliptic curve in Case II degenerates to a nodal curve. Here the special fibre $\CC_p$ of the regular minimal model consists of a genus $0$ curve $B$ and two chains of rational curves $X_i, 1 \leq i \leq 2n-1$ and $Y_j, 1 \leq j \leq 2m-1$. The closures of two of the Weierstrass points meet   $B$ as well  as  $X_n$ and $Y_m$. One has $X_i^2=Y_j^2=-2$ and $B^2=-4$.   

In this case there are essentially three different elements we can construct. Suppose $P$, $Q$ and $R$ are three Weierstrass points  whose closures lie on $B$, $X_n$ and $Y_m$ respectively, then one has the elements $\Xi_{P,Q}$, $\Xi_{P,R}$ and $\Xi_{Q,R}$. However, it is easy to see that $\Xi_{P,Q}-\Xi_{P,R} = \Xi_{Q,R}$ in $CH^2(X,1)$ as they differ by the tame symbol of a pair of functions.  

Suppose $P$ lies on $B$ and $Q$ lies on $X_n$. Then an analysis similar to what is done above shows, up to the boundary of a decomposable element,
$$\partial(\Xi_{P,Q})=2a \left(B+\sum_{j=1}^{2m-1} Y_j\right) - 2 \left( \sum_{i=1}^{n-1} (i-a) (X_i+X_{2n-i}) + (n-a)X_n \right)$$
Similarly, if $P$ lies on $B$ and $R$ lies on $Y_m$ one has 
$$\partial(\Xi_{P,R})=2a\left(B+\sum_{i=1}^{2n-1} X_i\right) - 2 \left( \sum_{j=1}^{m-1} (j-a) (Y_j+Y_{2m-j}) + (m-a)Y_m \right)$$
The boundary of a decomposable element $(C,p^a)$ is, like before, 
$$\partial((C,p^a))=a \left(B+\sum_{i=1}^{2n-1} X_i + \sum_{j=1}^{2m-1} Y_i \right).$$
Intersecting with $X_n$ and $Y_m$, for instance, shows that the boundaries are linearly independent and  hence they generate the group $PCH^1(Y)\otimes \Q$. So once again, the boundary map is surjective.

\subsubsection{Case IV}
$$
\includegraphics[width=80mm]{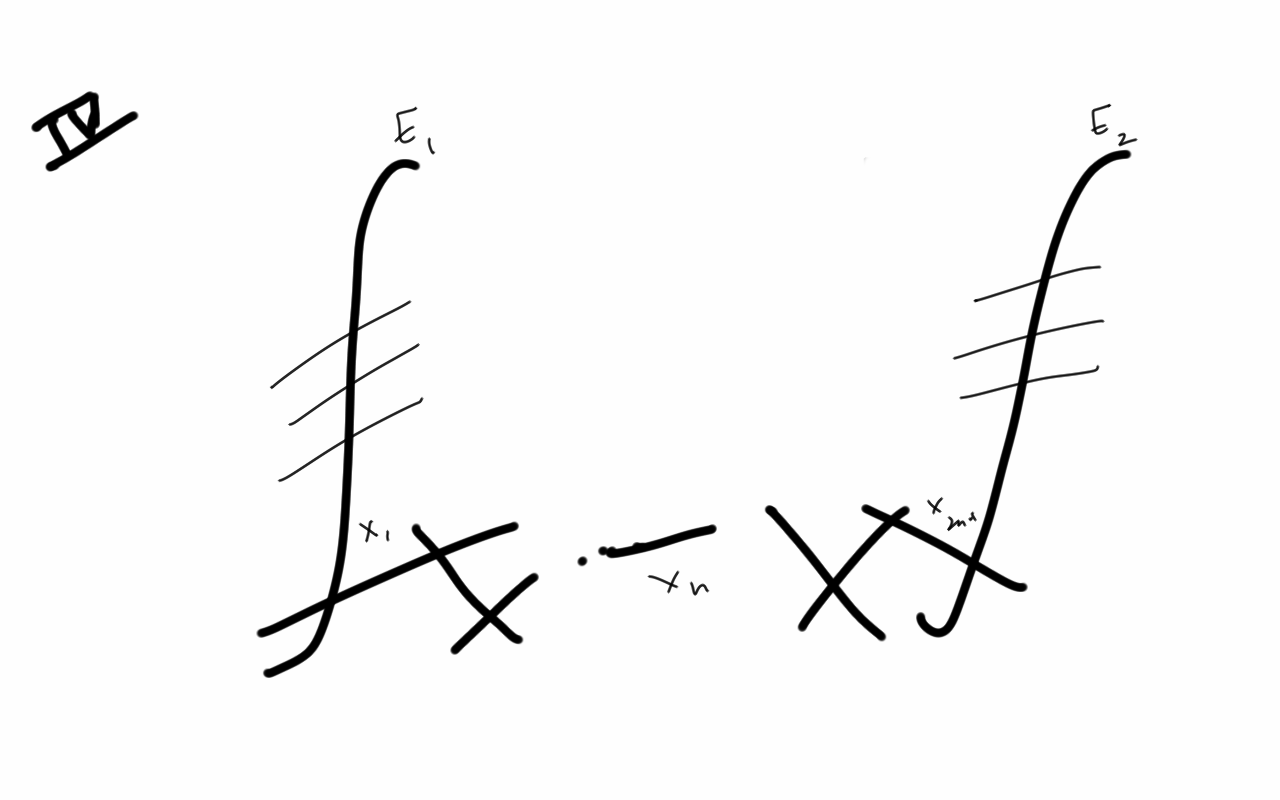}
$$
In this case the stable model is a union of two elliptic curves meeting at a point and the closures of three of the Weierstrass points lie on each elliptic curve. 
Here the Jacobian is smooth, hence is of Type 1. The regular minimal model consists of the two elliptic curves along with a chain of rational curves joining them. The elliptic curves $E_i$ satisfy $E_i^2=-1$ while the rational curves $X_j, 1 \leq j \leq r$  satisfy $X_j^2=-2$.

This was studied by Spiess \cite{spie} when the generic fibre product of elliptic curves as well.  While he constructed a particular element using an irreducible genus $2$ curve and two elliptic curves in the generic fibre, in fact one can use the element $\Xi_{P,Q}$ constructed above, with $P$ and $Q$ being chosen such that their closures lie on different components. Then if the divisor of $\bar{f}_P$  is 
$$\div(\bar{f}_P)=b_0 E_1+\sum_{i=1}^{r}  b_i X_i + b_{r+1} E_2$$
using a calculation  similar to that that above shows 
$$b_{i}=b_{0}-2i.$$
A convenient choice of $b_0$ is $r+1$ as in that case we have $b_i=(r+1-2i)=-b_{r+1-i}$ and the boundary is 
$$\partial(\Xi_{P,Q})=2(r+1) (E_1 - E_2 ) + \sum_{i=1}^{\lfloor \frac{r+1}{2}\rfloor} 2(r+1-2i) (X_i-X_{r+1-i}) $$
where $\lfloor a \rfloor$ denotes the greatest integer less than or equal to $a$.  Under the map to the Jacobian the $X_i's$ map to a point and so the cycle maps to $2(r+1)(E_1-E_2)$ which is clearly not a multiple of $E_1+E_2$.  Hence the boundary of the decomposable cycle, along with this cycle, generate the group $PCH^1(C_p)\otimes \Q$.

\subsubsection{Case V}
$$
\includegraphics[width=80mm]{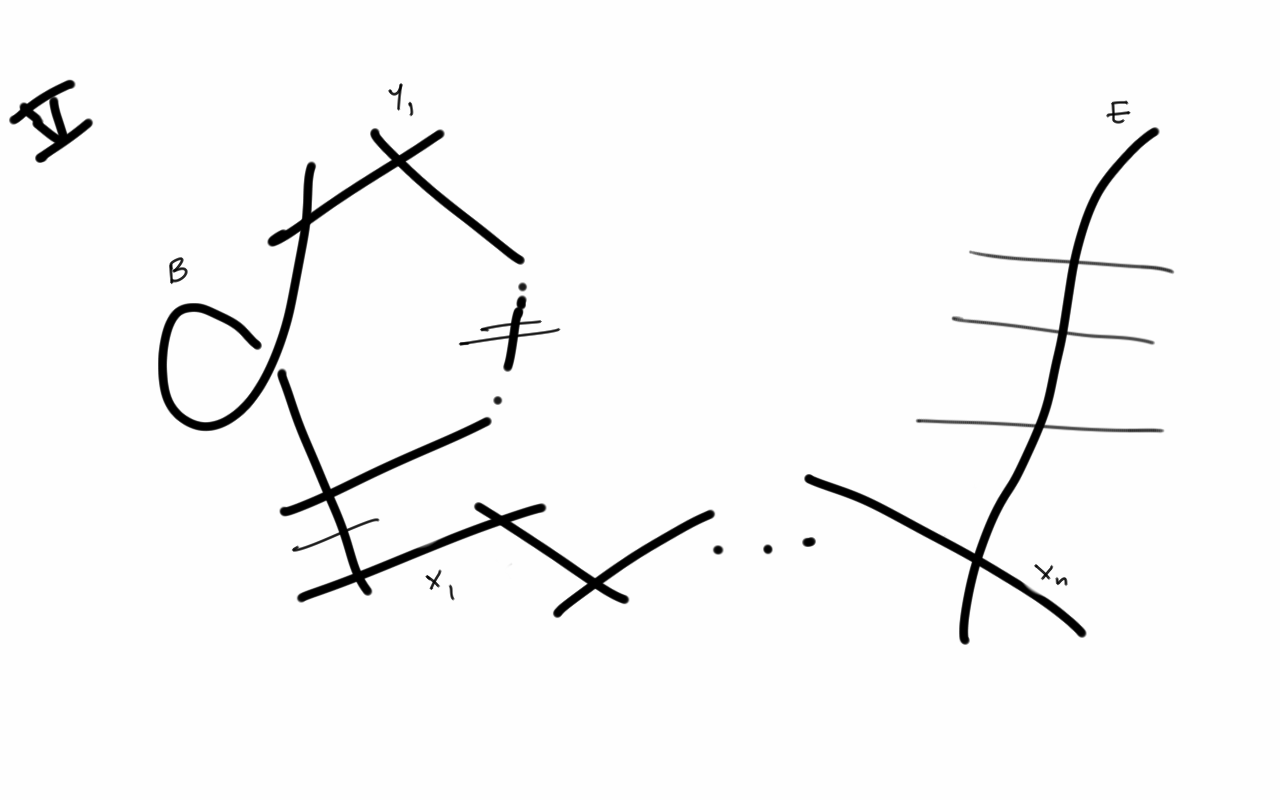}
$$
In this case the stable model is a union on an elliptic curve with a nodal rational curve and the Jacobian is of Type 2. The minimal regular model consists of the elliptic curve $E$ along with a chain of rational curves $X_i$ satisfying $X_i^2=-2$ linking $E$ with a rational curve $B$. One has $B^2=-3$ and $E^2=-1$.  Further, there is a chain of rational curves $Y_j, 1 \leq j \leq 2m-1$ linking two points of $B$ with $Y_j^2=-2$. This is essentially the case when one the elliptic curves in Case $IV$ degenerates to a nodal rational curve and corresponds to the case when the extension class of the elliptic surface is trivial - the surface is a product $E \times \CP^1$. 

Here three of the closures of the Weierstrass points lie on $E$, one the points lies on $B$ and finally two lie on $Y_m$.  We chose $P$ and $Q$ such that the closure of $P$ lies on $E$ and the closure of $Q$ lies on $B$. Calculating as before suppose 
$$\div(\barf_P)=\HH+aE+\sum_{i=1}^{r} b_iX_i + cB+\sum_{j=1}^{s} d_jY_j$$
one has $b_i=a-2i$, $c=a-2(r+1)$ and $d_j=a-2(r+1)$ for all $j$ where we use the fact that by symmetry $d_j=d_{s+1-j}$. So the boundary is 
$$\partial(\Xi_{P,Q})=4aE - 2\left(\sum_{i=1}^{r}  (2i-a))X_i + (2(r+1)-a)\left(B + \sum_{j=1}^{s} Y_j \right)\right)  $$

There is another element we can consider -- when $Q$ meets the component $Y_m$ instead of $B$ -- but a similar calculation shows that  the boundary is the same. If $P$ lies on $B$ and $Q$ lies on $Y_m$ or any such combination, the boundary can be seen to be $0$. 

\subsubsection{Case VI} 
$$
\includegraphics[width=80mm]{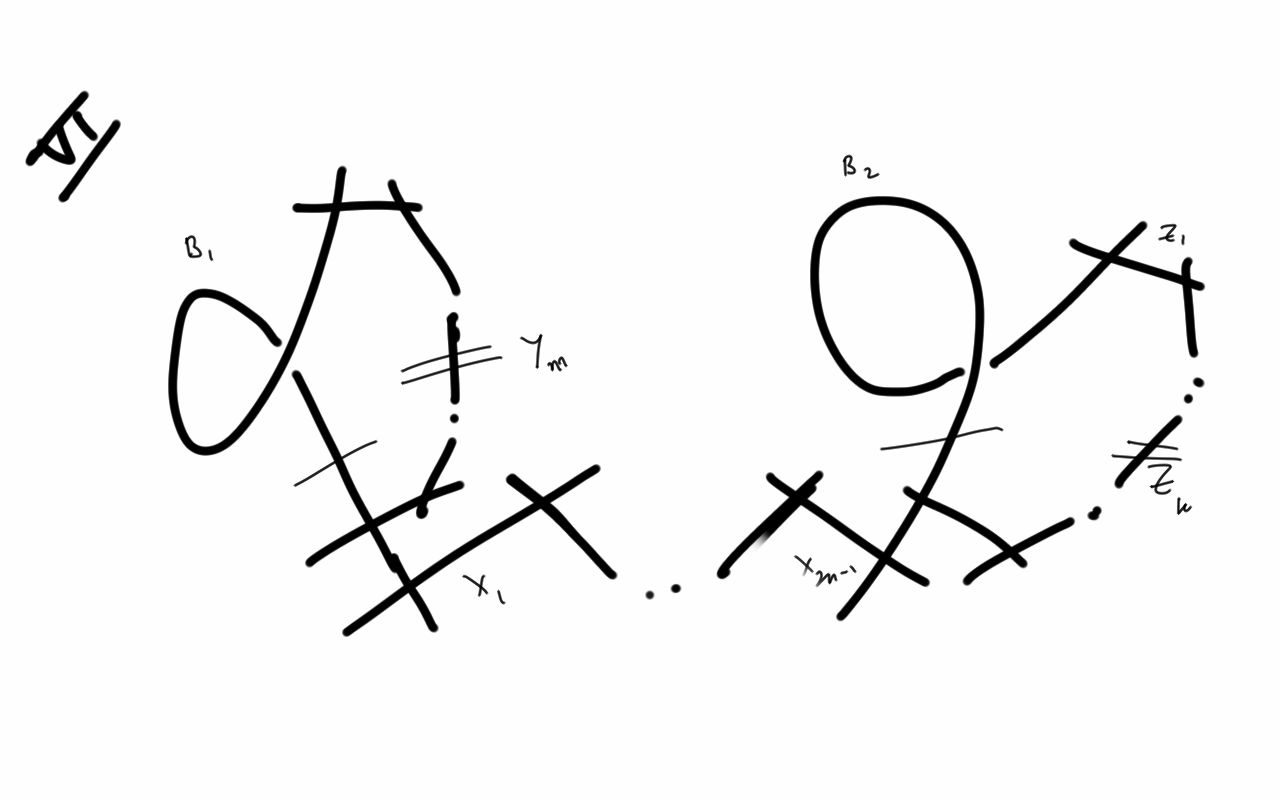}
$$
In this case the stable model is the union of two nodal rational curves meeting at a point. The minimal regular model consists of two rational curves $B_1$ and $B_2$ each with a chain of rational curves $Y_j, 1\leq j  \leq 2n-1$ and $Z_k, 1 \leq k \leq 2m-1$. The curves $B_i$ are also linked by a chain of rational curves $X_i, 1\leq i \leq s$. Finally, $X_i^2=Y_j^2=Z_k^2=-2$ and $B_1^2=B_2^2=-3$. A pair of the closures of the Weierstrass points meet the curves $Y_n$ and $Z_m$ each and the remaining two meet $B_1$ and $B_2$.

This is the case when both the  elliptic curves in Case IV degenerate to nodal rational curves. We studied this in \cite{sree2} - when the generic fibre was assumed to be a product of two non-isogenous elliptic curves. However, one can use the element above to prove surjectivity in more generality.  Here the Jacobian is of Type 3 so  one expects $2$ new elements. 

To get the first we choose $P$ and $Q$ such that their closures lie on $B_1$ and $B_2$ respectively. As before one has 
$$\div(\barf_P)=b_1B_1+ \sum_{i=1}^{s}  c_i X_i + \sum_{j=1}^{2n-1}  d_j Y_j + \sum_{k=1}^{2m-1}  e_k Z_k +b_2 B_2$$
Calculating as before, we have
$$-3b_1+c_1+d_1+d_{2n-1}+2=0$$
$$b_1-2d_1+d_2=0$$
In general we have 
$$d_{j+1}=2d_j-d_{j-1} \hspace{1in}    2 \leq j \leq 2n-2.$$
Hence, adding a decomposable element so that $b_1=0$, we see that $d_i=0$ for all $i$ and $c_1=-2$.  Further calculation shows that
 $$c_{i+1}=2c_i-c_{i-1} \hspace{1in}  2 \leq i \leq s-1$$ 
 Using this we have $c_i=-2i$. We also have  
 $$b_2-2c_s+c_{s-1}=0$$ 
 which shows that $b_2=-2(s+1)$. Finally, using symmetry to say that $e_i=e_{2m-i}$, we have
$$-3b_2+c_s+2e_1 -2=0$$
$$e_{k-1}-2e_k+e_{k+1}=0 \hspace{1in} 1<k<2m-1$$
$$e_{2m-2}-2e_{2m-1}+b_2=0$$
which shows that $e_k=-2(s+1)$ for all $k$. Hence the divisor is 
$$\div(\barf_P)=\sum_{i=1}^{s} (-2i) X_i -2(s+1) \left(B_2+\sum_{k=1}^{2m-1} Z_k\right)$$
Finally,  adding a decomposable element to get a more symmetric expression we see that 
$$\partial (\Xi_{P,Q})=2\left(( s+1) (B_1+\sum_{j=1}^{2n-1} Y_j) \right)  + 2 \left( \sum_{i=1}^{\lfloor \frac{s+1}{2}\rfloor} (s+1-2i) (X_i-X_{s+1-i})\right)  - 2 \left((s+1) (B_2+\sum_{k=1}^{2m-1} Z_k)\right)$$ 

To get the second new element we choose $P$ and $Q$ such that the closure of $P$ lies on $B_1$ as before but the closure of $Q$ lies on $Z_{m}$. As before we can assume $b_1=0$ and the same calculation as above holds to show $d_j=0$ and $c_i=-2i$ and $b_2=-2(s+1)$. 

The first difference is that we have 
$$c_{s}-3b_2+2e_1=0$$ 
hence 
$$e_1=-2(s+1)-1.$$
Using
$$b_2-2e_1+e_2=0$$ 
$$e_{k+1}=2e_{k}-e_{k-1} \hspace{1in} 1\leq k \leq m-1$$
$$e_{m-k}=e_{m+k} \hspace{1in} 1\leq k \leq m-1$$
shows $e_k=-2(s+1)-k$ for $1\leq k \leq m$ and so one has 
$$\div(\barf_{P})=\sum_{i=1}^{s} -2i X_i -2(s+1) B_2 -\left (\sum_{k=1}^{m-1} (2(s+1)+k) (Z_k+Z_{m+k} )+(2(s+1)+m) Z_m \right)$$

\subsubsection{Case VII}
$$
\includegraphics[width=80mm]{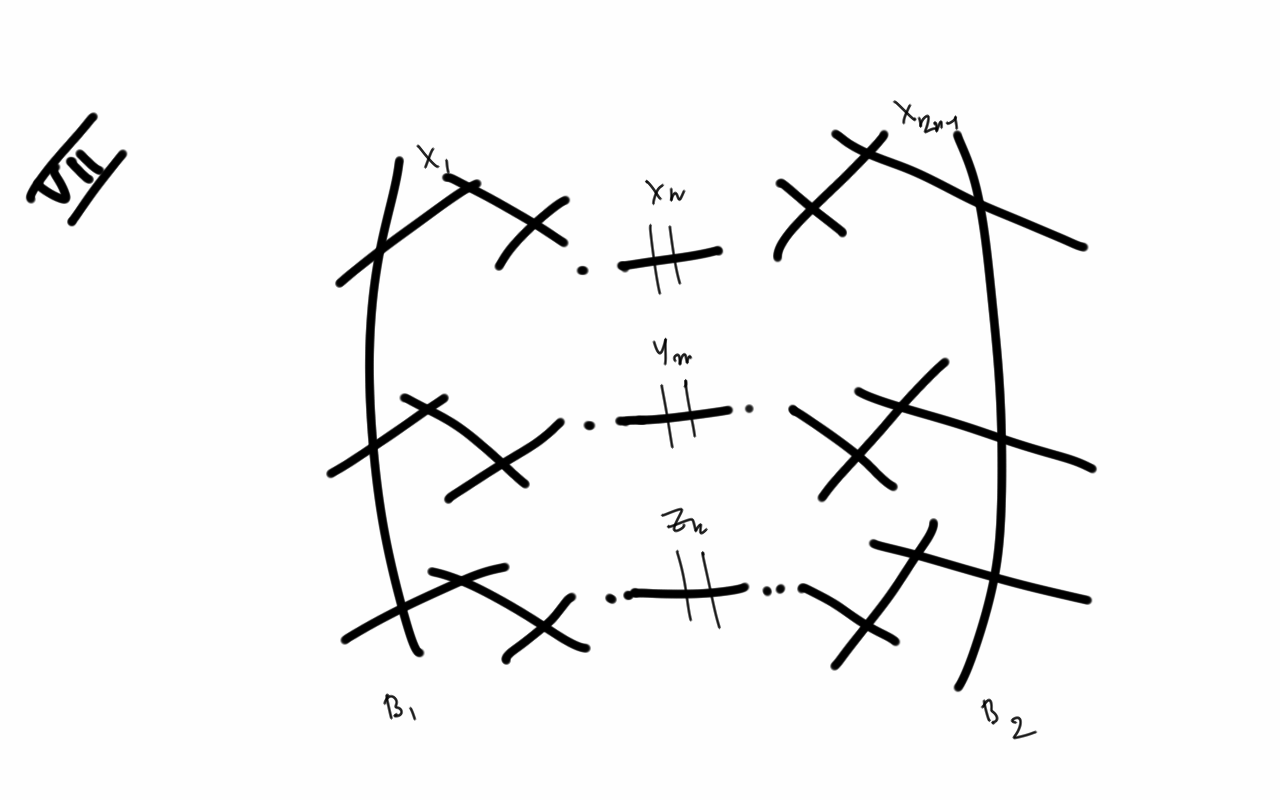}
$$
In this case the stable model of the  special fibre consists of two smooth rational curves $B_1$ and $B_2$ meeting at three points. The minimal regular model consists of the two smooth curves along with three chains of rational curves $X_i$ $1\leq i \leq 2r-1$, $Y_j$ $1 \leq j \leq 2s-1$ and  $Z_k$, $1 \leq k \leq 2t-1$ lying over the three points of intersection.  A pair each of the  closures of the Weierstrass points meet the curves $X_r$, $Y_s$ and $Z_t$. One has $B_i^2=-3$ and $X_i^2=Y_j^2=Z_k^2=-2$. 

Once again the Jacobian is of Type 3 and so one expects two new elements. To get the first we choose Weierstrass points $P$, $Q$ and $R$ such that the closure of $P$ lies on $X_r$, the closure of $Q$ lies on $Y_s$ and the closure of $R$ lies on $Z_t$. Then, if $f_P$ is as before, we have  
$$\div(\barf_P)=b_1B_1+b_2B_2+\sum a_i X_i + \sum c_j Y_j + \sum d_k Z_k$$ 
Assume $b_1=0$. Then one has $a_1+c_1+d_1=0$. Intersecting the divisor  with $X_1$ shows that  $b_1-2a_1+a_2=0$. Continuing in this manner, intersecting with the $X_i$s one gets 
$$a_i=\begin{cases} ia_1 & 1\leq i \leq r\\ ia_1-2(i-r) & r \leq i \leq 2r-1 \end{cases}$$
$$b_2=2ra_1-2r$$
Similarly, using $Y_j$s one gets 
$$c_j=\begin{cases} jc_1 & 1\leq j \leq s\\
jc_1+2(j-s) &  s \leq j \leq 2s-1\end{cases}$$
$$b_2=2sc_1+2s$$
and using $Z_k$s one has 
$$d_k=kd_1 \hspace{\baselineskip} 1\leq i \leq 2k-1$$
$$b_2=2kd_1$$
Solving this system of simultaneous equations shows $a_1=1, c_1=-1$ and $b_1=b_2=d_1=0$. Hence one has 
$$\partial (\Xi_{P,Q}) =2 \left( rX_r  + \sum_{i=1}^{r-1} i(X_i+X_{2r-i})  \right) - 2\left(  sY_s+ \sum_{j=1}^{s-1}  j(Y_j+Y_{2s-j}) \right)$$
A similar calculation works for the elements $\Xi_{Q,R}$ and $\Xi_{P,R}$. Clearly one has 
$$\Xi_{P,R}=\Xi_{P,Q}+\Xi_{Q,R}$$
hence there are essentially two elements one can construct in this manner. 

To check that these elements are lineary independent we intersect with $X_r$, $Y_s$ and $Z_t$. $(\partial (\Xi_{P,Q}).X_r)=-2$ while $(X_r.\partial \Xi_{Q,R})=0$ hence they are not linearly equivalent. Further $\partial(\Xi_{P,Q}) \neq 0$. Similarly, intersecting with $Y_s$ shows that $\partial (\Xi_{Q,R}) \neq 0$.

\bibliographystyle{alpha}
\bibliography{referencesABSS}

\noindent Ramesh Sreekantan\\ 
Statistics and Mathematics Unit\\
Indian Statistical Institute\\
$8^{th}$ Mile, Mysore Road\\
Jnanabharathi, Bangalore, 560 059\\
Karnataka, India \\
\\
Email: rameshsreekantan\@@gmail.com

\end{document}